\documentclass[11pt]{article}

\usepackage[margin=1in]{geometry}
\usepackage{amsmath,amssymb,amsthm}
\usepackage{bm}
\usepackage{booktabs}
\usepackage{xcolor}
\usepackage{microtype}
\usepackage{url}
\usepackage[hidelinks]{hyperref}

\theoremstyle{plain}
\newtheorem{theorem}{Theorem}
\newtheorem{lemma}[theorem]{Lemma}
\newtheorem{proposition}[theorem]{Proposition}
\newtheorem{corollary}[theorem]{Corollary}
\theoremstyle{definition}
\newtheorem{definition}{Definition}
\newtheorem{assumption}{Assumption}
\theoremstyle{remark}

\DeclareMathOperator{\dist}{dist}
\DeclareMathOperator{\Var}{Var}
\newcommand{\SCV}{\mathrm{SCV}}

\newcommand{\ind}{\mathbf{1}}
\newcommand{\PH}{\mathrm{PH}}
\newcommand{\ME}{\mathrm{ME}}
\newcommand{\E}{\mathbb{E}}
\newcommand{\e}{\mathrm{e}}
\newcommand{\ii}{\mathrm{i}}
\newcommand{\dd}{\mathrm{d}}
\newcommand{\bfalp}{\bm{\alpha}}
\newcommand{\bfs}{\bm{s}}
\newcommand{\Z}{\mathbb{Z}}
\newcommand{\R}{\mathbb{R}}
\newcommand{\N}{\mathbb{N}}
\newcommand{\eqd}{\overset{d}{=}}

\title{Optimization-Free Concentrated Matrix-Exponentials}
\author{Maria Laura Battagliola \and Oscar Peralta}
\date{}

\begin{document}
\maketitle

\begin{abstract}
Near-deterministic positive delays require highly concentrated distributions, but phase-type models are constrained by the Erlang variance limit. While matrix-exponential distributions can empirically bypass this barrier, prior low-variance constructions relied entirely on numerical optimization. We propose an explicit family of concentrated matrix-exponential densities for the unit delay, obtained by raising the trigonometric Fej\'er kernel to logarithmic power. With exact moments and closed-form parameters, this gives the first analytical proof of a matrix-exponential class that asymptotically surpasses the Erlang bound.
\end{abstract}


\section{Introduction}\label{sec:intro}
Approximating nearly deterministic and strictly positive delays requires probability distributions with very small squared coefficient of variation $\SCV(X)=\Var(X)/\E[X]^2$, which equals $0$ for a deterministic delay and is invariant under scaling. We therefore take the deterministic target to be the unit delay, without loss of generality. Within the phase-type (PH) class, the theorem of Aldous and Shepp states that, among all order-$n$ laws $X\in\PH(n)$, the Erlang law is least variable, with $\SCV(X)=1/n$ \cite{aldous1987}; approximating deterministic delays by Erlang laws is accordingly called Erlangization.
The matrix-exponential (ME) class is a broader rational-transform framework: ME laws have rational Laplace transforms and finite-dimensional moment formulas, yet can achieve substantially lower variability than PH laws of the same order. Numerical concentrated-ME methods \cite{horvath2016,horvath2020,meszaros2022} showed by high-dimensional optimization that ME distributions with complex-conjugate poles can far surpass the Erlang bound at many finite orders; this has also been exploited in numerical Laplace inversion \cite{horvathnilt2020}. Whether this superiority persists asymptotically remained open: each low-SCV construction required a finite-$n$ numerical search.

This paper introduces the powered-{Fej\'er} concentrated matrix-exponential (PF-CME) family. The {Fej\'er} kernel $\Phi_m(\cdot)$ is a nonnegative trigonometric polynomial of order $m$ and provides a natural approximate identity inside the common-damping ME class. By raising this kernel to the logarithmically growing power $r_m=\lceil\log m\rceil$ and choosing $h_m=2\log m+\log\log m$, we obtain an explicit ME family whose SCV decays, up to constant factors, as $\log m/m^2$ in the construction index and as $\log^3 n_m/n_m^2$ in the corresponding ME dimension $n_m$. This is the first such concentration law for a common-damping ME family, confirming that the numerical superiority over Erlang in \cite{horvath2016,horvath2020,meszaros2022} is structural. The precise concentration statement is given in Theorem~\ref{thm:main}.

Beyond this asymptotic result, the explicit parameters of the PF-CME family are directly applicable to numerical inverse Laplace transformation (NILT). Its density is nonnegative by construction, avoiding the Gibbs-type overshoots in some classical inversion formulas \cite{horvathnilt2020}. Practically, the family is asymptotically superior, but its SCV remains above the Erlang bound below order about $7260$. Thus PF-CME complements, rather than replaces, Erlang and optimized CME families at moderate orders.

Throughout we use standard asymptotic notation: $f(m)=O(g(m))$ means
$|f(m)|\le Cg(m)$ eventually for some $C>0$;
$f(m)=\Theta(g(m))$, equivalently $f(m)\asymp g(m)$, means
$c g(m)\le f(m)\le Cg(m)$ eventually for some $c,C>0$;
$f(m)=o(g(m))$ means $f(m)/g(m)\to0$; and $f(u)\propto g(u)$
denotes proportionality up to a positive factor.

\section{The common-damping framework}\label{sec:background}

A phase-type law is the absorption time of a finite-state Markov jump process. If $X\sim\PH(n)$, its density takes the form $f(t)=\bfalp\,\e^{St}(-S{\bm{1}})$ with sub-intensity matrix $S$ and initial distribution $\bfalp$; positivity is automatic, and the Aldous--Shepp bound gives
\begin{equation}\label{eq:erlangbound}
  \inf\bigl\{\SCV(X):X\in\PH(n)\bigr\}=\frac{1}{n}.
\end{equation}
The class $\ME(n)$ replaces $S$ with a general real matrix $A$, requiring only a real, nonnegative density that integrates to $1$ \cite[Chapters 3--4]{bladt2017}. A density belongs to the ME class if and only if its Laplace transform is rational; its minimal algebraic order is the degree of the rational transform. In an order-$n$ representation it can be written as $f(t)=\bfalp\,\e^{At}\bfs$ with moments $\mu_k=k!\bfalp(-A)^{-(k+1)}\bfs$. Because $\ME(n)\supset\PH(n)$ strictly, the bound \eqref{eq:erlangbound} does not apply, and substantially lower SCV values are achievable. Matrix-exponential distributions numerically optimized to approach this deterministic limit are termed concentrated matrix-exponential (CME) laws.

If an ME density has only simple poles $-\beta$ and $-\beta\pm\ii\ell\omega$ ($\ell=1,\dots,L$), all sharing the same real part, then
\begin{equation}\label{eq:oscillatoryform}
  f(t)=\e^{-\beta t}\Bigl(c_0+\sum_{\ell=1}^{L}\bigl(c_\ell\cos(\ell\omega t)+d_\ell\sin(\ell\omega t)\bigr)\Bigr),
\end{equation}
an exponentially damped trigonometric polynomial, with moment integrals reducing via $\int_0^\infty t^p\e^{-(\beta-\ii\ell\omega)t}\dd t=p!(\beta-\ii\ell\omega)^{-(p+1)}$. 
Optimized CME laws often exhibit this common real part after normalization \cite{horvath2020,horvath2018}, motivating the following analytically tractable subclass.

\begin{assumption}[Common-damping framework]\label{ass:commondamping}
Fix $L\in\N$. After scaling so that the common real part equals $-1$, consider densities of the form
\begin{equation}\label{eq:cdform}
  f(t)=C\,\e^{-t}\,G(\omega t+\phi),\qquad t\ge 0,
\end{equation}
where $\omega>0$, $\phi\in\R$, $C>0$, and $G$ is a nonnegative trigonometric polynomial of degree $L$.
\end{assumption}

The class \eqref{eq:cdform} has one real pole at $-1$ and at most $L$ conjugate pairs at $-1\pm\ii\ell\omega$, so its algebraic order is at most $2L+1$, with equality if every frequency appears with nonzero coefficient.

\section{The PF-CME family}\label{sec:family}

For $m\ge 2$, the Fej\'er kernel of order $m$ is
\begin{equation}\label{eq:Fejer}
  \Phi_m(\theta)=1+2\sum_{\ell=1}^{m-1}\!\left(1-\frac{\ell}{m}\right)\cos(\ell\theta)
  =\frac{1}{m}\!\left(\frac{\sin(m\theta/2)}{\sin(\theta/2)}\right)^{\!2}.
\end{equation}
It is nonnegative, has degree $m-1$, and attains its maximum of $m$ at $\theta=0$ \cite[Vol. I, Chap. III, Sec. 3]{zygmund2002trigonometric}.

Throughout set
\begin{equation}\label{eq:params}
  r_m:=\lceil\log m\rceil,\quad h_m:=2\log m+\log\log m,\quad \omega_m:=\frac{2\pi}{h_m}.
\end{equation}

\begin{definition}[PF-CME density]\label{def:PFfamily}
The powered-Fej\'er concentrated matrix-exponential density of index $m\ge 2$ is
\begin{equation}\label{eq:densitypf}
  f_m(t):=C_m\,\e^{-t}\,W_m\!\bigl(\omega_m(t-1)\bigr),\qquad t\ge 0,
\end{equation}
where ${W_m(\theta):=\bigl(\Phi_m(\theta)\bigr)^{r_m}}$ and $C_m>0$.
We write $X_m\sim\mathrm{PF\mbox{-}CME}_m$ for the random variable with density $f_m$.
\end{definition}

The shift places the principal peak at $t=1$, the unit deterministic delay used here without loss of generality, while $\omega_m=2\pi/h_m$ makes $[0,h_m)$ one full modulation period. Normalization need not impose $\E[X_m]=1$ exactly, but our main result in Theorem~\ref{thm:main} gives $\E[X_m]\to1$. We now identify the ME order.

\begin{proposition}[ME order]\label{prop:order}
For each $m\ge2$, $X_m\sim\mathrm{PF\mbox{-}CME}_m$ has minimal ME order $n_m:=2r_m(m-1)+1$.
\end{proposition}
\begin{proof}

In complex form, the Fej\'er kernel satisfies
\[
  \Phi_m(\theta)
  =\frac1m\left|\sum_{j=0}^{m-1}\e^{\ii j\theta}\right|^2
  =\sum_{|k|\le m-1}\left(1-\frac{|k|}{m}\right)\e^{\ii k\theta}=\sum_{{k\in\mathbb Z}}{\gamma_k}\e^{\ii k\theta}.
\]
Here ${\gamma_k}=1-|k|/m$ for $|k|\le m-1$ and ${\gamma_k}=0$ otherwise. Since
$W_m=\Phi_m^{r_m}$, pointwise multiplication of trigonometric polynomials
corresponds to discrete convolution of their Fourier sequences. Hence
$W_m$ is a nonnegative trigonometric polynomial of degree
$L_m:=r_m(m-1)$ with cosine expansion

\begin{equation}\label{eq:Wexpansion}
  W_m(\theta)=B_{0,m}+2\sum_{\ell=1}^{L_m}B_{\ell,m}\cos(\ell\theta),
\end{equation}
where $B_{\ell,m}$ is the $\ell$-th coefficient of the $r_m$-fold discrete
convolution of ${(\gamma_k)}$.
Substituting $\theta=\omega_m(t-1)$ and applying 
$\cos(a-b)=\cos a\cos b+\sin a\sin b$ yields
\begin{equation}\label{eq:fexpanded}
  f_m(t)=C_m\e^{-t}\Bigl[B_{0,m}+2\sum_{\ell=1}^{L_m}B_{\ell,m}
  \bigl(\cos(\ell\omega_m)\cos(\ell\omega_m t)+\sin(\ell\omega_m)\sin(\ell\omega_m t)\bigr)\Bigr],
\end{equation}
a linear combination of $\e^{-t}$, $\e^{-t}\cos(\ell\omega_m t)$, and 
$\e^{-t}\sin(\ell\omega_m t)$ for $\ell=1,\dots,L_m$, whose Laplace transforms 
share the denominator $(s+1)\prod_{\ell=1}^{L_m}((s+1)^2+\ell^2\omega_m^2)$ 
of degree $n_m$. Since convolution preserves strict positivity on gap-free finite supports, all $B_{\ell,m}>0$, $0\le \ell\le L_m$. After the shift, the complex coefficient at frequency $\ell$ is $B_{\ell,m}\e^{-\ii\ell\omega_m}\neq0$, so every pole pair has nonzero residue and no cancellation occurs. Thus the Laplace transform has degree exactly $n_m$.
\end{proof}

The moments and SCV use the basis integrals
\begin{align}
  J_0(a)&=\int_0^\infty \e^{-t}\cos(at)\,\dd t=\frac{1}{1+a^2}, &
  S_0(a)&=\int_0^\infty \e^{-t}\sin(at)\,\dd t=\frac{a}{1+a^2},\label{eq:JS0}\\
  J_1(a)&=\int_0^\infty t\e^{-t}\cos(at)\,\dd t=\frac{1-a^2}{(1+a^2)^2}, &
  S_1(a)&=\int_0^\infty t\e^{-t}\sin(at)\,\dd t=\frac{2a}{(1+a^2)^2},\label{eq:JS1}\\
  J_2(a)&=\int_0^\infty t^2\e^{-t}\cos(at)\,\dd t=\frac{2(1-3a^2)}{(1+a^2)^3}, &
  S_2(a)&=\int_0^\infty t^2\e^{-t}\sin(at)\,\dd t=\frac{2a(3-a^2)}{(1+a^2)^3}.\label{eq:JS2}
\end{align}

\begin{theorem}[Exact moments]\label{thm:moments}
For $k=0,1,2$, define
\[
  M_k(m):=B_{0,m}J_k(0)+2\sum_{\ell=1}^{L_m}B_{\ell,m}\Bigl(\cos(\ell\omega_m)J_k(\ell\omega_m)+\sin(\ell\omega_m)S_k(\ell\omega_m)\Bigr).
\]
Then $C_m=M_0(m)^{-1}$, $\E[X_m]=M_1(m)/M_0(m)$, $\E[X_m^2]=M_2(m)/M_0(m)$, and $\SCV(X_m)=\tfrac{M_2(m)M_0(m)}{M_1(m)^2}-1$.
\end{theorem}

\begin{proof}
Integrate \eqref{eq:fexpanded} against $1$, $t$, and $t^2$ term by term using \eqref{eq:JS0}--\eqref{eq:JS2}.
\end{proof}
\section{Asymptotics and comparisons}\label{sec:comparison}
We first derive the law for $\SCV(X_m)$ in terms of $m$, then translate it using $r_m\asymp\log m$ and $n_m\asymp m\log m$.

\begin{lemma}[Geometric cell decomposition]\label{lem:geom}
Set $q_m:=\e^{-h_m}=1/(m^2\log m)$  from \eqref{eq:params}, let $K_m\sim\mathrm{Geom}(1-q_m)$ on $\{0,1,2,\dots\}$, and let $Y_m$ have the density $g_m(u)\propto\e^{-u}W_m(\omega_m(u-1))\ind_{[0,h_m)}(u)$.
Then $X_m\eqd Y_m+K_m h_m$ with $Y_m$ and $K_m$ independent, and
\begin{align}
  \E[X_m]&=\E[Y_m]+\frac{h_m q_m}{1-q_m},\quad
  \Var(X_m)=\Var(Y_m)+\frac{h_m^2 q_m}{(1-q_m)^2}.\label{eq:var_decomp}
\end{align}
\end{lemma}

\begin{proof}
For $t=kh_m+u$ with $k\in\N_0$ and $u\in[0,h_m)$, the periodicity $W_m(\omega_m(t-1))=W_m(\omega_m(u-1)+2\pi k)=W_m(\omega_m(u-1))$ gives
\[
  f_m(t)=C_m\,\e^{-t}\,W_m(\omega_m(t-1))=C_m\,\e^{-kh_m}\,\e^{-u}\,W_m(\omega_m(u-1))\,=\,C_m\,q_m^k\,\e^{-u}\,W_m(\omega_m(u-1)).
\]
Summing over $k\ge0$, the marginal density of $t\bmod h_m$ is proportional to $\e^{-u}W_m(\omega_m(u-1))$ on $[0,h_m)$, which is exactly $g_m$. Integrating over the $k$-th cell $[kh_m,(k+1)h_m)$ gives
\[
  \int_{kh_m}^{(k+1)h_m}f_m(t)\,\dd t = q_m^k\int_0^{h_m}C_m\e^{-u}W_m(\omega_m(u-1))\,\dd u = (1-q_m)q_m^k,
\]
Since $\int_0^{h_m}C_m\e^{-u}W_m(\omega_m(u-1))\,\dd u=1-q_m$, this is the $\mathrm{Geom}(1-q_m)$ mass at $k$. The factorization yields $X_m=Y_m+K_mh_m$ with $Y_m$ and $K_m$ independent, and \eqref{eq:var_decomp} follows from geometric moments.
\end{proof}

Since $h_m\asymp\log m$ and $q_m=1/(m^2\log m)$, the geometric variance term in \eqref{eq:var_decomp} has order $\log m/m^2$. It remains to estimate the base-cell variance; we first need kernel bounds.
\begin{lemma}[Scaled kernel bounds]\label{lem:scaledfejer}
Let $\dist(x,2\pi\Z):=\min_{k\in\Z}|x-2\pi k|$. There exist constants
$0<a_1<a_2<\infty$ and $m_0$ such that for all $m\ge m_0$:
\begin{align}
  m\,\e^{-a_2 m^2x^2}\le \Phi_m(x)\le m\,\e^{-a_1 m^2x^2},
  &\qquad |x|\le \tfrac{5}{m}\quad\text{(peak)},\label{eq:localF}\\
  \Phi_m(x)\le \tfrac{6}{25}m,
  &\qquad \tfrac{5}{m}\le |x|\le \tfrac{20}{m}\quad\text{(transition)},\label{eq:midF}\\
  \Phi_m(x)\le \tfrac{\pi^2}{400}m,
  &\qquad \tfrac{20}{m}\le \dist(x,2\pi\Z)\le \pi\quad\text{(tail)}.\label{eq:farF}
\end{align}
\end{lemma}

\begin{proof}
By $2\pi$-periodicity, it suffices to work on $[-\pi,\pi]$ and prove the tail bound for $20/m\le |x|\le\pi$. Write $\Phi_m(x)/m
  =
  \left(\sin(y/2)/(y/2)\right)^2 \,
  \left((y/(2m))/\sin(y/(2m))\right)^2$, $y=mx.$
On $|y|\le5$, the quotient $\sin(y/2)/(y/2)$ is positive and has no zeros.
$H(y):=-\log\left(\sin(y/2)/(y/2)\right)^2$, $H(0)=0$, is even and analytic on $(-2\pi,2\pi)$. Since $\left(\sin(y/2)/(y/2)\right)^2=1-y^2/12+O(y^4)$
and $-\log(1-z)=z+O(z^2)$, we have $H(y)=y^2/12+O(y^4)$ near $0$.
By compactness, $H(y)\asymp y^2$ on $|y|\le5$, so there are $b_1,b_2>0$ with
$\e^{-b_2y^2}
  \le
  \left(\sin(y/2)/(y/2)\right)^2
  \le
  \e^{-b_1y^2}$, $|y|\le5.$
 As for the second factor, $\left((y/(2m))/\sin(y/(2m))\right)^2 = 1+O(y^2/m^2)$ 
uniformly on $|y|\le5$,
so, adjusting the constants and taking $m$ large enough,
$m\e^{-a_2m^2x^2}\le \Phi_m(x)\le m\e^{-a_1m^2x^2}$, $|x|\le5/m.$
This proves \eqref{eq:localF}.

The remaining two estimates follow from the same representation. First suppose
$5/m\le |x|\le20/m$, equivalently $5\le |y|\le20$. Using the universal bound
$|\sin(\cdot)|\le1$,
$\left(\sin(y/2)/(y/2)\right)^2\le 4/y^2\le4/25$.
Since $|y/(2m)|=o(1)$ uniformly in this region, the second factor
$\left((y/(2m))/\sin(y/(2m))\right)^2$ is at most $3/2$ for large $m$. Thus,
for $5/m\le |x|\le20/m$, $\Phi_m(x)\le(6/25)m$, proving \eqref{eq:midF}.
Finally, for $20/m\le |x|\le\pi$, the inequality $\sin u\ge 2u/\pi$ on
$0\le u\le\pi/2$, with $u=|x|/2$, gives
$|\sin(x/2)|\ge |x|/\pi\ge20/(\pi m)$. Hence, for $20/m\le |x|\le\pi$,
$\Phi_m(x)\le \frac{1}{m\sin^2(x/2)}
\le \frac{\pi^2}{m|x|^2} \le \frac{\pi^2 m}{400}$,
proving \eqref{eq:farF}.
\end{proof}

We apply Lemma~\ref{lem:scaledfejer} on ${[-\omega_m,2\pi-\omega_m]}$: the central peak is leading, while transition and tail parts are negligible.

\begin{theorem}[Base-cell concentration]\label{thm:base_var}
As $m\to\infty$, $\Var(Y_m)=\Theta(\log m/m^2)$ and $\E[Y_m]=1+O(\log m/m^2)$.
\end{theorem}
\begin{proof}
Write $r=r_m$, $\omega=\omega_m$, and ${W=W_m}$. With $x=\omega(u-1)$,
\[
  A_m=\frac{\e^{-1}}{\omega}\int_{-\omega}^{2\pi-\omega}\e^{-x/\omega}W(x)\,\dd x,\qquad
  N_{2,m}=\frac{\e^{-1}}{\omega^3}\int_{-\omega}^{2\pi-\omega}x^2\e^{-x/\omega}W(x)\,\dd x,
\]
so $\E[(Y_m-1)^2]=N_{2,m}/A_m$. Set
$I_0=\{x\in[-\omega,2\pi-\omega]: |x|\le5/m\}$, $I_1=\{x\in[-\omega,2\pi-\omega]: 5/m<|x|\le20/m\}$
and $I_2=\{x\in[-\omega,2\pi-\omega]: |x|>20/m\}$. Since $\omega\asymp1/\log m$, for large $m$ we have $20/m<\omega$, so $I_0$ and $I_1$ are full symmetric intervals around the origin.

On $I_0$, Lemma~\ref{lem:scaledfejer} gives, after raising to the power $r$,
$W(x)=\Theta(m^r\e^{-{\kappa}rm^2x^2})$ in the two-sided Gaussian sense.
Since $|x|/\omega\le5/(m\omega)=o(1)$, $\e^{-x/\omega}$ is uniformly $1+o(1)$ and does not change the order. With the change of variables $y=m\sqrt r\,x$, we obtain
\[
  \int_{I_0}\e^{-x/\omega}W(x)\,\dd x
  =\Theta\!\left({\frac{m^r}{m\sqrt r}
    \int_{-5\sqrt r}^{5\sqrt r}\e^{-{\kappa}y^2}\,\dd y}\right)
  =\Theta\!\left(\frac{m^r}{m\sqrt r}\right),
\]
\[
  \int_{I_0}x^2\e^{-x/\omega}W(x)\,\dd x
  =\Theta\!\left({\frac{m^r}{m^3r^{3/2}}
    \int_{-5\sqrt r}^{5\sqrt r}y^2\e^{-{\kappa}y^2}\,\dd y}\right)
  =\Theta\!\left(\frac{m^r}{m^3r^{3/2}}\right).
\]
$I_0$ supplies the leading mass and second moment about $1$.

On $I_1$, \eqref{eq:midF} gives $W(x)\le(6m/25)^r$. Since $|I_1|=O(1/m)$ and $x^2=O(m^{-2})$ on this region,
\[
\begin{aligned}
  \int_{I_1}\e^{-x/\omega}W(x)\,\dd x
  &\le C\,|I_1|\left(\frac{6m}{25}\right)^r
  =O\!\left(\frac{1}{m}\left(\frac{6m}{25}\right)^r\right)
  =O\!\left(m^{r-1}\left(\frac{6}{25}\right)^r\right),\\
  \int_{I_1}x^2\e^{-x/\omega}W(x)\,\dd x
  &\le C\,m^{-2}|I_1|\left(\frac{6m}{25}\right)^r
  =O\!\left(\frac{1}{m^3}\left(\frac{6m}{25}\right)^r\right)
  =O\!\left(m^{r-3}\left(\frac{6}{25}\right)^r\right).
\end{aligned}
\]
On $I_2$, points with $x\le\pi$ satisfy $\dist(x,2\pi\Z)=|x|>20/m$, while for $x>\pi$ we have $x\le2\pi-\omega$, hence $\dist(x,2\pi\Z)=2\pi-x\ge\omega>20/m$ by the choice of large $m$. Hence the tail bound in Lemma~\ref{lem:scaledfejer} applies throughout $I_2$, giving $W(x)\le(\pi^2m/400)^r$. Also, $I_2$ has bounded length, $x^2=O(1)$, and $\e^{-x/\omega}\le\e$. Thus
\[
  \int_{I_2}(1+x^2)\e^{-x/\omega}W(x)\,\dd x \le C\,|I_2|\left(\frac{\pi^2m}{400}\right)^r =O\!\left(m^r\left(\frac{\pi^2}{400}\right)^r\right).
\]
Relative to the $I_0$ scales, the largest transition and tail errors are bounded by $(6/25)^r r^{3/2}$ and $(\pi^2/400)^r m^3r^{3/2}$ for the second moment, and by $(6/25)^r\sqrt r$ and $(\pi^2/400)^r m\sqrt r$ for the mass. All vanish because $r=\lceil\log m\rceil$, $(\pi^2/400)^r m^3 \le m^{-\delta}$ for some $\delta>0$, and $6/25<\e^{-1}$. Therefore
\[
  A_m=\Theta\!\left(\frac{m^r}{\omega m\sqrt r}\right),\qquad
  N_{2,m}=\Theta\!\left(\frac{m^r}{\omega^3m^3r^{3/2}}\right),\quad\mbox{and}\quad
  \E[(Y_m-1)^2]=\Theta\!\left(\frac{1}{\omega^2m^2r}\right)=\Theta\!\left(\frac{\log m}{m^2}\right).
\]
For the mean, the unperturbed first moment on $I_0$ cancels because  $W(\cdot)$ is even: $\int_{I_0}xW(x)\,\dd x=0$. Hence
\[
  \int_{I_0}x\e^{-x/\omega}W(x)\,\dd x
  =
  \int_{I_0}x(\e^{-x/\omega}-1)W(x)\,\dd x.
\]
Using the Taylor bound $|\e^{-x/\omega}-1|\le C|x|/\omega$ on $I_0$ and the second-moment estimate above gives
\[
  \int_{I_0}x\e^{-x/\omega}W(x)\,\dd x=O\!\left(\frac{m^r}{\omega m^3r^{3/2}}\right).
\]
The transition and tail first moments satisfy
\[
  \int_{I_1}|x|\e^{-x/\omega}W(x)\,\dd x=O\!\left(m^{r-2}\left(\frac{6}{25}\right)^r\right),\qquad
  \int_{I_2}|x|\e^{-x/\omega}W(x)\,\dd x=O\!\left(m^r\left(\frac{\pi^2}{400}\right)^r\right).
\]
These are $o(m^r/(\omega m^3r^{3/2}))$: for $I_1$ this uses $6/25<\e^{-1}$, and for $I_2$ it uses $\pi^2/400<\e^{-3}$. Thus
\[
  \E[Y_m]-1=\frac{1}{\omega}
  \frac{O(m^r/(\omega m^3r^{3/2}))}{\Theta(m^r/(m\sqrt r))}
  =O\!\left(\frac{1}{\omega^2m^2r}\right)=O\!\left(\frac{\log m}{m^2}\right).
\]
Finally, $\Var(Y_m)=\E[(Y_m-1)^2]-(\E[Y_m]-1)^2$, and $(\E[Y_m]-1)^2=O(\log^2 m/m^4)=o(\log m/m^2)$. Thus $\Var(Y_m)=\Theta(\log m/m^2)$.
\end{proof}

\begin{theorem}[Concentration law]\label{thm:main}
Let $X_m\sim\mathrm{PF\mbox{-}CME}_m$, whose minimal ME order is $n_m=2r_m(m-1)+1$. Then $ \Var(X_m)=\Theta\!\left(\frac{\log m}{m^2}\right)$ and $\E[X_m]=1+O\!\left(\frac{\log m}{m^2}\right)$,
and consequently $\SCV(X_m)=\Theta(\log m/m^2)=\Theta(\log^3 n_m/n_m^2)$.
\end{theorem}

\begin{proof}
By Lemma~\ref{lem:geom}, $\Var(X_m)=\Var(Y_m)+h_m^2q_m/(1-q_m)^2=\Theta(\log m/m^2)$, using Theorem~\ref{thm:base_var} and $q_m=1/(m^2\log m)$. Also
$\E[X_m]=\E[Y_m]+h_mq_m/(1-q_m)=1+O(\log m/m^2)$, hence $\SCV(X_m)=\Theta(\log m/m^2)$. Finally, $n_m\asymp m\log m$ gives $\log m/m^2\asymp\log^3 n_m/n_m^2$.
\end{proof}

\begin{corollary}\label{cor:diagnostic}
Since $\log^3 n_m/n_m^2=o(1/n_m)$, PF-CME asymptotically surpasses Erlang. Moreover, for $X_m\sim\mathrm{PF\mbox{-}CME}_m$,
$(m^2/\log m)\,\SCV(X_m)$ and $(n_m^2/\log^3 n_m)\,\SCV(X_m)$ stay bounded away from $0$ and $\infty$.
\end{corollary}

Theorem~\ref{thm:moments} gives the crossover between $m=519$ and $m=520$.
\begin{center}
\vspace{-0.4em}
\refstepcounter{table}\label{tab:scv_comparison}
{\footnotesize\textbf{Table~\thetable.} PF-CME family: {$\SCV(X_m)$}, ratio $\SCV(X_m)\cdot n_m$ to the Erlang bound
(below $1$ means PF-CME beats Erlang), and scaling diagnostics.}
\par\vspace{0.15em}
{\scriptsize
\setlength{\tabcolsep}{2.7pt}\renewcommand{\arraystretch}{0.86}
\begin{tabular}{@{}rrrcccc@{}}
\toprule
$m$ & $r_m$ & $n_m$ & {$\SCV(X_m)$} & $\SCV(X_m)\cdot n_m$ 
& $(m^2/\log m)\SCV(X_m)$ & $(n_m^2/\log^3 n_m)\SCV(X_m)$ \\
\midrule
200 & 6 & 2389 & $8.003{\times}10^{-4}$ & 1.91 & 6.04 & 9.70 \\
400 & 6 & 4789 & $2.265{\times}10^{-4}$ & 1.08 & 6.05 & 8.53 \\
519 & 7 & 7253 & $1.379{\times}10^{-4}$ & 1.000 & 5.94 & 10.33 \\
\textbf{520} & 7 & 7267 & $1.374{\times}10^{-4}$ & \textbf{0.999} & 5.94 & 10.33 \\
1000 & 7 & 13987 & $4.106{\times}10^{-5}$ & 0.574 & 5.94 & 9.23 \\
2000 & 8 & 31985 & $1.112{\times}10^{-5}$ & 0.356 & 5.85 & 10.19 \\
\bottomrule
\end{tabular}}
\vspace{-0.6em}
\end{center}

\section{Conclusion}\label{sec:conclusion}

PF-CME is the first closed-form, optimization-free common-damping ME construction that provably surpasses the Erlang bound. For each $m$, all parameters and exact moments are explicit. This settles an open question from prior numerical studies \cite{horvath2016,horvath2020,horvathnilt2020,meszaros2022}: along the subsequence ${n=n_m}$, Theorem~\ref{thm:main} shows that ${\SCV(X_m)<1/n_m}$ is structural (indeed ${\SCV(X_m)=\Theta(\log^3 n_m/n_m^2)}$), not an artifact of high-dimensional search. Although optimization-based methods achieve lower SCVs at moderate orders, PF-CME gives a rigorous asymptotic guarantee and a benchmark for localized numerical refinement. Whether a purely closed-form CME family achieving $\SCV=\Theta(n^{-2})$ exists remains an open question.


\begin{thebibliography}{99}

\bibitem{aldous1987}
David Aldous and Larry Shepp.
\newblock The least variable phase type distribution is Erlang.
\newblock \emph{Stochastic Models}, 3(3):467--473, 1987.
\newblock doi: \url{https://doi.org/10.1080/15326348708807067}.

\bibitem{horvath2016}
Ill{\'e}s Horv{\'a}th, Orsolya S{\'a}f{\'a}r, Mikl{\'o}s Telek, and Bence Z{\'a}mb{\'o}.
\newblock Concentrated matrix exponential distributions.
\newblock In D. Fiems, M. Paolieri, and A. Platis, editors, \emph{Computer Performance Engineering. EPEW 2016}, Lecture Notes in Computer Science, volume 9951, pages 18--31. Springer, Cham, 2016.
\newblock doi: \url{https://doi.org/10.1007/978-3-319-46433-6_2}.

\bibitem{horvath2020}
G{\'a}bor Horv{\'a}th, Ill{\'e}s Horv{\'a}th, and Mikl{\'o}s Telek.
\newblock High order concentrated matrix-exponential distributions.
\newblock \emph{Stochastic Models}, 36(2):176--192, 2020.
\newblock doi: \url{https://doi.org/10.1080/15326349.2019.1702058}.

\bibitem{meszaros2022}
Andr{\'a}s M{\'e}sz{\'a}ros and Mikl{\'o}s Telek.
\newblock Concentrated matrix exponential distributions with real eigenvalues.
\newblock \emph{Probability in the Engineering and Informational Sciences}, 36(4):1171--1187, 2022.
\newblock doi: \url{https://doi.org/10.1017/S0269964821000309}.

\bibitem{horvathnilt2020}
G{\'a}bor Horv{\'a}th, Ill{\'e}s Horv{\'a}th, Salah Al-Deen Almousa, and Mikl{\'o}s Telek.
\newblock Numerical inverse Laplace transformation using concentrated matrix exponential distributions.
\newblock \emph{Performance Evaluation}, 137:102067, 2020.
\newblock doi: \url{https://doi.org/10.1016/j.peva.2019.102067}.

\bibitem{bladt2017}
Mogens Bladt and Bo Friis Nielsen.
\newblock \emph{Matrix-Exponential Distributions in Applied Probability}.
\newblock Probability Theory and Stochastic Modelling, volume 81. Springer New York, NY, 2017.
\newblock doi: \url{https://doi.org/10.1007/978-1-4939-7049-0}.

\bibitem{horvath2018}
Ill{\'e}s Horv{\'a}th, Zs{\'o}fia Talyig{\'a}s, and Mikl{\'o}s Telek.
\newblock An optimal inverse Laplace transform method without positive and negative overshoot -- an integral based interpretation.
\newblock \emph{Electronic Notes in Theoretical Computer Science}, 337:87--104, 2018.
\newblock doi: \url{https://doi.org/10.1016/j.entcs.2018.03.035}.

\bibitem{zygmund2002trigonometric}
Antoni Zygmund.
\newblock \emph{Trigonometric Series}.
\newblock 3rd edition, volumes I and II combined. Cambridge University Press, Cambridge, UK, 2002.

\end{thebibliography}
\end{document}